\documentclass[a4paper, 10pt, twoside, notitlepage]{amsart}

\usepackage[utf8]{inputenc}
\usepackage{color}
\usepackage{amsmath} 
\usepackage{amssymb} 
\usepackage{amsthm}
\usepackage{geometry}
\usepackage{graphicx}
\usepackage{esint}
\usepackage{subcaption}
\usepackage[colorlinks=true,linkcolor=blue]{hyperref}

\newcounter{statement}
\newtheorem{prop}[statement]{Proposition}

\newtheorem{rmk}[statement]{Remark}
\newtheorem{lem}[statement]{Lemma}
\newtheorem{cor}[statement]{Corollary}
\newtheorem{thm}[statement]{Theorem}

\newcommand {\R} {\mathbb{R}} 
\newcommand {\T} {\mathbb{T}} \newcommand {\N} {\mathbb{N}}

\newcommand {\p} {\partial}

\newcommand{\intd}{\, \mathrm{d}} 	
\newcommand{\eps}{\varepsilon}	

\def\Xint#1{\mathchoice
{\XXint\displaystyle\textstyle{#1}}%
{\XXint\textstyle\scriptstyle{#1}}%
{\XXint\scriptstyle\scriptscriptstyle{#1}}%
{\XXint\scriptscriptstyle\scriptscriptstyle{#1}}%
\!\int}
\def\XXint#1#2#3{{\setbox0=\hbox{$#1{#2#3}{\int}$ }
\vcenter{\hbox{$#2#3$ }}\kern-.58\wd0}}

\def\dashint{\Xint-}

\title[Rigidity for the Four-Well Problem]{On Rigidity for the Four-Well Problem Arising in the Cubic-to-Trigonal Phase Transformation}

\author{Angkana Rüland}
\address{Institut für Angewandte Mathematik,
Ruprecht-Karls-Universität Heidelberg,
Im Neuenheimer Feld 205,
69120 Heidelberg}
\email{angkana.rueland@uni-heidelberg.de}

\author{Theresa M. Simon}
\address{Angewandte Mathematik Münster: Institut für Analysis und Numerik,
Fachbereich Mathematik und Informatik der Universität Münster,
Orléans-Ring 10,
48149 Münster}
\email{theresa.simon@uni-muenster.de}

\begin{document}
\maketitle
\begin{abstract}
We classify all exactly stress-free solutions to the cubic-to-trigonal phase transformation within the geometrically linearized theory of elasticity, showing that only simple laminates and crossing-twin structures can occur. In particular, we prove that although this transformation is closely related to the cubic-to-orthorhombic phase transformation, all its solutions are rigid.
The argument relies on a combination of the Saint-Venant compatibility conditions together with the underlying nonlinear relations and non-convexity conditions satisfied by the strain components.
\end{abstract}

\section{Introduction}

Shape-memory alloys are materials with a thermodynamically very interesting behaviour: They undergo a diffusionless, solid-solid phase transformation in which symmetry is reduced. More precisely, a highly symmetric high temperature phase, the \emph{austenite}, transforms into a much less symmetric low temperature phase, the \emph{martensite}, upon cooling below a certain critical temperature \cite{B}. Mathematically, these materials have very successfully been described by an energy minimization \cite{BJ92} of the form 
\begin{align}
\label{eq:min}
\int\limits_{\Omega} W(\nabla u, \theta) \intd x  \rightarrow \min.
\end{align}
Here $\Omega \subset \R^3$ denotes the reference configuration, which often is chosen to be the austenite state at the critical temperature $\theta_c>0$, the deformation of the material is $u: \Omega \rightarrow \R^3$, the temperature is denoted by $\theta: \Omega \rightarrow [0,\infty)$ and $W:  \R^{3\times 3}_+ \times [0,\infty) \rightarrow [0,\infty)$ corresponds to the stored energy function. This function encodes the physical properties of the material and is assumed to be
\begin{itemize}
\item[(i)] \emph{frame indifferent}, i.e. $W(F,\theta) = W(QF,\theta)$ for all $F \in \R^{3\times 3}_{+}$ and $Q\in SO(3)$,
\item[(ii)] \emph{invariant with respect to the material symmetry}, i.e. $W(F,\theta) = W(F H, \theta)$ for $H\in \mathcal{P}_{a}$ where $\mathcal{P}_{a}$ denotes the symmetry group of the austenite phase, which we assume to strictly include the symmetry group of the martensite phase.
\end{itemize}
Here (i) can be viewed as a geometric nonlinearity, while (ii) encodes the main material nonlinearity which, for instance, reflects the transition from the highly symmetric austenite to the less symmetric martensite phase. Both structure conditions imply that the energies in \eqref{eq:min} are highly non-quasiconvex and thus give rise to a rich energy landscape. As a result, minimizing sequences can be rather intricate, which physically leads to various different microstructures.

In this note, it is our objective to study a specific phase transformation for which experimentally interesting microstructures are observed. Seeking to capture ``crossing-twin structures'' in a fully three-dimensional model (see Figure \ref{fig:laminates}, left), we focus on the \emph{cubic-to-trigonal phase transformation} in three dimensions. This deformation, for instance, arises in materials such as Zirconia or in Cu-Cd-alloys but also in the cubic-to-monoclinic transformation in CuZnAl. We refer to \cite{S97,PS69} for experimental studies, to \cite{HS00} for an investigation of special microstructures in a geometrically nonlinear context and to \cite{BD00, BD01} for mathematical relaxation results for the associated geometrically nonlinear problems.
Since the study of the minimization problem \eqref{eq:min} can be rather complex, in this note we make the following three simplifying assumptions which are common in the mathematical analysis of martensitic phase transformations: 
\begin{itemize}
\item We \emph{fix} temperature below the transition temperature, 
\item we consider only the \emph{material} nonlinearity while \emph{linearizing} the geometric nonlinearity,
\item and we study only \emph{exactly stress-free} structures.
\end{itemize} 
Instead of investigating the full minimization problem \eqref{eq:min}, we thus study the differential inclusion
\begin{align}
\label{eq:diff_incl}
e(u) \in \{e^{(1)}, e^{(2)}, e^{(3)}, e^{(4)}\} \mbox{ in } \T^3,
\end{align}
where 
\begin{align}
\label{eq:strains}
\begin{split}
e^{(1)}= \begin{pmatrix} d_1 & 1 & 1\\ 1 & d_2 & 1 \\ 1 & 1 & d_3 \end{pmatrix}, \ e^{(2)} = \begin{pmatrix} d_1 & -1 & -1\\ -1 & d_2 & 1 \\ -1 & 1 & d_3 \end{pmatrix},\\
e^{(3)} = \begin{pmatrix} d_1 & 1 & -1\\ 1 & d_2 & -1 \\ -1 & -1 & d_3 \end{pmatrix}, \  \ e^{(4)} = \begin{pmatrix} d_1 & -1 & 1\\ -1 & d_2 & -1 \\ 1 & -1 & d_3 \end{pmatrix},
\end{split}
\end{align}
and $d_1,d_2,d_3$ are material-specific constants. In order to avoid additional mathematical difficulties, we assume that the reference configuration is given by the torus $\T^3 := \T_1\times\T_2\times \T_3$, where $\T_i := [0,\lambda_i)$ for some $\lambda_i >0$ and $i \in \{1,2,3\}$.

 We observe that all the matrices in \eqref{eq:strains} are symmetrized rank-one connected, i.e. for each $i,j\in\{1,2,3,4\}$ there exist $a_{ij}\in \R^3 \setminus \{0\}, n_{ij}\in S^2$ such that
\begin{align*}
e^{(i)}-e^{(j)} = \frac{1}{2}(a_{ij}\otimes n_{ij} + n_{ij}\otimes a_{ij}).
\end{align*}
It is well-known that, as a consequence, the differential inclusion \eqref{eq:diff_incl} thus allows for so-called \emph{twin} or \emph{simple laminate} solutions, i.e. solutions $u(x) = u(n_{ij}\cdot x)$ with $n_{ij}\in S^2$ denoting the vectors from above. These are rather rigid, one-dimensional structures, which are frequently observed in experiments \cite{B}, see also Figure \ref{fig:laminates}, right. The possible pairs $(a_{ij}, n_{ij}) \in \R^3 \times S^2$ for the cubic-to-trigonal phase transformation are collected in Table \ref{tab:laminates}.

Contrary to other materials such as alloys undergoing a cubic-to-tetragonal phase transformation, simple laminates are not the only possible solutions to \eqref{eq:diff_incl}. As in the (more complex) cubic-to-orthorhombic phase transformation, also in the cubic-to-trigonal phase transformation ``crossing-twin structures'' can emerge. These are two-dimensional structures involving ``laminates within laminates'' (see Figure \ref{fig:laminates}, left). In particular, these patterns locally consist of zero-homogeneous deformations which involve specific ``corners'' which are formed by four different variants of martensite.

\begin{table}
\begin{tabular}{l|l}
\textbf{strains} & \textbf{possible normals $(a_{ij},n_{ij})$}\\
\hline
$e^{(1)}, e^{(2)}$ & $[1,0,0], [0,4,4]$\\
$e^{(1)}, e^{(3)}$ & $[0,0,1], [4,4,0]$\\
$e^{(1)}, e^{(4)}$ & $[0,1,0], [4,0,4]$\\
$e^{(2)}, e^{(3)}$ & $[0,1,0], [-4,0,4]$\\
$e^{(2)}, e^{(4)}$ & $[0,0,1], [-4,4,0]$\\
$e^{(3)}, e^{(4)}$ & $[1,0,0], [0,4,-4]$\\
\end{tabular}
\medskip
\caption{The possible normals arising in simple laminate constructions. For these the strain alternates between the two strain values given in the left column of the table.}
\label{tab:laminates}
\end{table}

\subsection{The main result}

As our main result, we classify all solutions to the differential inclusion \eqref{eq:diff_incl} and prove that in addition to the simple laminate solutions only crossing twin structures arise.

\begin{thm}
\label{thm:structure}
Let $e(u):= \frac{1}{2}(\nabla u + (\nabla u)^t)$ with $e(u): \T^3 \rightarrow \R^{3\times 3}_{sym}$ be a periodic symmetrized gradient. Assume that
\begin{align*}
e(u) \in \{e^{(1)}, e^{(2)}, e^{(3)}, e^{(4)}\}.
\end{align*}
Then the following structure result holds:
\begin{itemize}
\item[(i)] There exists $j\in\{1,2,3\}$ such that $\p_j e(u)=0$. 
\item[(ii)] Assuming that $j = 2$, there exist functions $f_1: \T_1 \rightarrow \R$ and $f_3:\T_3 \rightarrow \R $ such that $f_1, f_{3} \in \{-1,1\}$ and either $e_{13}(u)(x) = f_1(x_1)$ for a.e. $x\in \Omega$ or $e_{13}(u)(x) = f_3(x_3)$ for a.e. $x\in \Omega$.
\item[(iii)] Assuming that $e_{13}(u)(x) = f_3(x_3)$, consider $\Phi(s,t):= (t -F_3(s),s)$, where $F_3(s)$ is such that $F_3'(s) = f_3(s)$ a.e.~and $F_3(0)=0$. Then, there exists $g: \Phi^{-1}(\T^3) \rightarrow \R$, $(s,t) \mapsto g(t)$ such that
\begin{align*}
(e_{12}\circ \Phi)(s,t) = g(t), \ (e_{23}\circ \Phi)(s,t) = f_3(s) g(t).
\end{align*}
\end{itemize}
\end{thm}

\begin{rmk}
	All cases not listed result from the symmetries of the model under permutation of the space directions, as these only permute the side lengths of the torus $\T^3$ and the constants $d_1$, $d_2$, and $d_3$, the precise values of these constants do not enter the argument. The permutations play the following roles:
	If an index $i\in \{1,2,3\}$ has been fixed, the other two can be exchanged via transposition.
	A fixed index can be transformed into a different fixed index by a full cyclic permutation.
\end{rmk}

Let us comment on this result: From a materials science point of view, it gives a complete classification of exactly stress-free solutions for the cubic-to-trigonal phase transformation in the geometrically linear framework.
Mathematically, Theorem \ref{thm:structure} provides a rigidity result for a phase transformation which leads to more complex structures than simple laminates. While a similar classification and rigidity result had been obtained in \cite{R16} for the cubic-to-orthorhombic phase transformation in three dimensions, this required strong geometric assumptions on the smallest possible scales.
These assumptions were also necessary as a result of the presence of convex integration solutions for the corresponding differential inclusion in the case of the cubic-to-orthorhombic phase transformation. In contrast, in our model, such conditions are \emph{not} needed. Due to the smaller degrees of freedom that are present (four instead of six possible strains), in fact \emph{any} exactly stress-free solution must satisfy the structural conditions and ``wild'' convex integration solutions are ruled out.

\begin{figure}
\centering
\begin{subfigure}{0.55 \textwidth}
\centering
\includegraphics[width =  \textwidth]{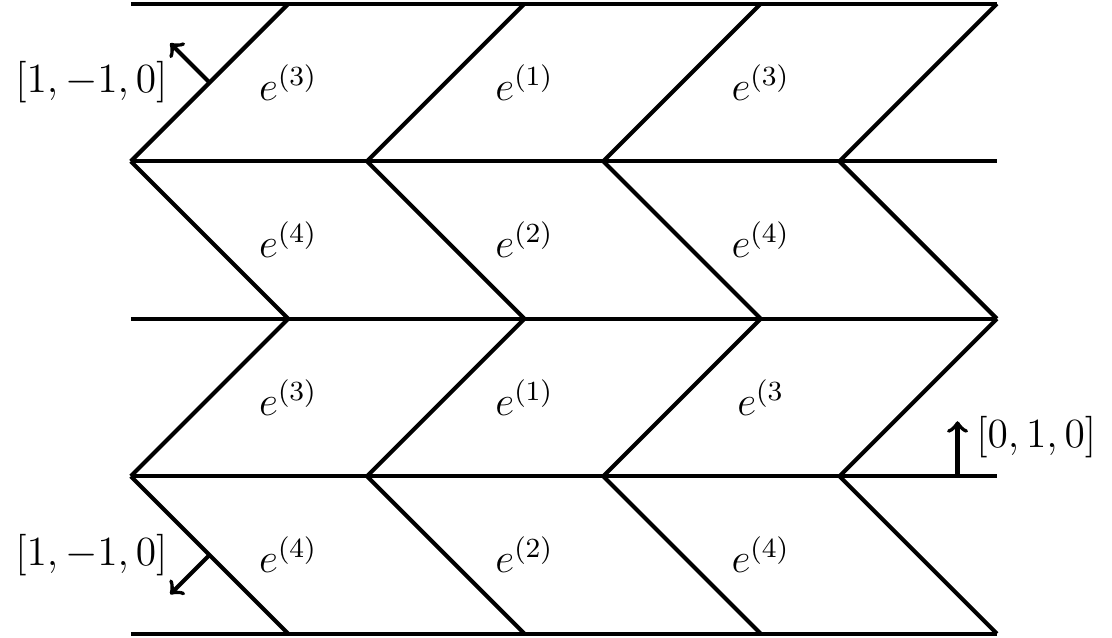}
\end{subfigure}
\hspace{0.2cm}
\begin{subfigure}{0.25 \textwidth}
\centering
\includegraphics[width =  \textwidth]{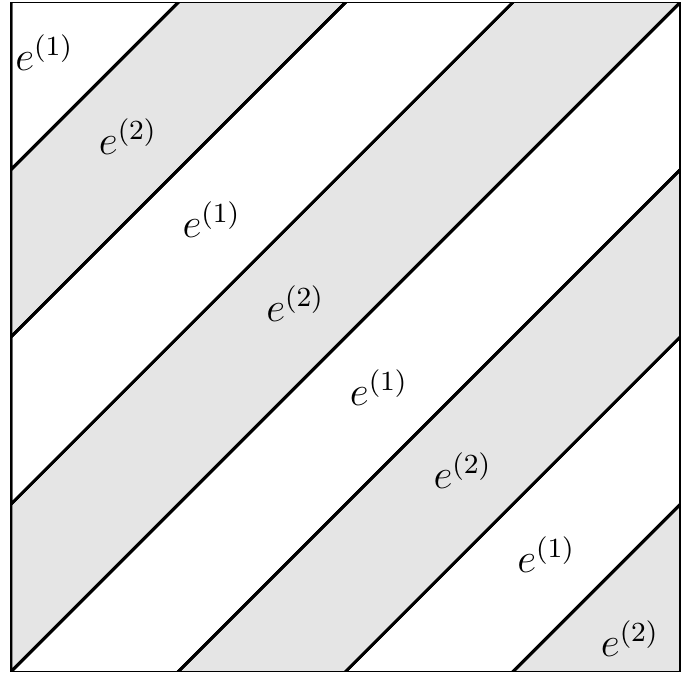}
\end{subfigure}
\caption{Schematic illustration of a crossing twin structure (left) and a simple laminate (right). Only very specific twins can be used to form crossing twin structures. These are obtained as a consequence of the compatibility conditions satisfied by the strain equations.}
\label{fig:laminates}
\end{figure}

\subsection{Relation to the cubic-to-orthorhombic phase transformation}
Due to the outlined rather different behaviour (in terms of rigidity and flexibility) of stress-free solutions of the cubic-to-orthorhombic and the cubic-to-trigonal phase transformation, we explain the algebraic relation between these two transformations:
To this end, we recall that for the cubic-to-orthorhombic phase transformation, the exactly stress-free setting in the geometrically linearized situation corresponds to the differential inclusion
\begin{align*}
e( u) \in \{e^{(1)},\dots, e^{(6)}\},
\end{align*}
with 
\begin{align*}
e^{(1)}&:= \begin{pmatrix} 1 & \delta & 0 \\ \delta & 1 & 0\\ 0 & 0 & -2 \end{pmatrix}, \ 
e^{(2)}:= \begin{pmatrix} 1 & -\delta & 0 \\ -\delta & 1 & 0\\ 0 & 0 & -2 \end{pmatrix}, \
e^{(3)}:= \begin{pmatrix} 1 & 0 & \delta \\ 0 & -2 & 0\\ \delta & 0 & 1 \end{pmatrix}, \\ 
e^{(4)}&:= \begin{pmatrix} 1 & 0 & -\delta \\ 0 & -2 & 0\\ -\delta & 0 & 1 \end{pmatrix},\
e^{(5)}:= \begin{pmatrix} -2 & 0 & 0 \\ 0 & 1 & \delta\\ 0 & \delta & 1 \end{pmatrix}, \ 
e^{(6)}:= \begin{pmatrix} -2 & 0 & 0 \\ 0 & 1 & -\delta\\ 0 & -\delta & 1 \end{pmatrix}.
\end{align*}
Here $\delta>0$ is a material dependent parameter.
If now one assumes that a microstructure only involves the infinitesimal strains $\{e^{(1)},\dots,e^{(4)}\}$ and if one carries out the change of coordinates $x\mapsto \hat{x}:=C^{-t}x$, $u\mapsto \hat{u}:= C u$ with

\begin{align*}
C:= 
\begin{pmatrix} \frac{1}{\sqrt{3}} & 0 & 0 \\ 0 & \frac{\sqrt{3}}{\sqrt{2} \delta} & 0 \\ 0 & 0 &\frac{1}{\sqrt{3}} \end{pmatrix} \cdot \begin{pmatrix} 0 & 1 & 1\\ \sqrt{2} & 0 & 0 \\ 0 & 1 & -1 \end{pmatrix},
\end{align*}
using that the (infinitesimal) strain transforms according to $e(\hat{u}) = C e(u) C^t$, one exactly arrives at the differential inclusion \eqref{eq:diff_incl} for the cubic-to-trigonal phase transformation with the parameters $d_1 = - \frac{1}{3}$, $d_2 = \frac{3}{\delta^2}$, $d_3 = - \frac{1}{3}$.
This shows that the differential inclusion for the cubic-to-trigonal phase transformation indeed corresponds to a subset of the differential inclusion for the cubic-to-orthorhombic phase transformation. Due to the fewer degrees of freedom, in contrast to the full cubic-to-orthorhombic phase transformation, it however displays strong rigidity properties.

\subsection{Main ideas}
The arguments for the proof of Theorem \ref{thm:structure} rely on a combination of the linear compatibility conditions for strains in the form of the Saint-Venant equations and the nonlinear constraints in our model. More precisely, the Saint-Venant conditions imply structural conditions on the possible space dependences of the strains.
Furthermore, the full classification result requires a breaking of symmetries that can only be deduced in combination with the non-convexity of the problem, i.e., the fact that for all $i,j \in \{1,2,3\}$ we have that $e_{ij}( u)$ attains at most three possible values and the nonlinear relation $e_{23}-e_{12}e_{13}=0$.
For a simplified model with only two-dimensional dependences similar arguments had earlier been considered in \cite{R16}. However, in contrast to \cite{R16}, in the present setting we do \emph{not} need to make use of the additional structural condition of two-dimensionality. Using restrictions to carefully chosen planes, as the key part of our argument, we in fact \emph{infer} the two-dimensionality of the strains and then combine this with the ideas from \cite{R16}.

\subsection{Relation to the literature}

In the study of minimization problems of the type \eqref{eq:min} a common first step consists of the analysis of exactly stress-free structures. This is investigated for particular low energy nucleation problems in \cite{CKZ17,CDPRZZ19}, for the two-well problem in \cite{DM1,DM2}, for the cubic-to-tetragonal phase transformation in \cite{K,CDK} and for the cubic-to-orthorhombic transformation in \cite{R16,Rue16}. Moreover, rigidity properties of related differential inclusions without gauge symmetries are studied in \cite{K,CK02,Z98, Sverak, T93, KMS03, Pompe}.
For sufficiently complex structures of the energy minima, a striking dichotomy between rigidity of the underlying exactly stress-free structures under relatively high regularity conditions (e.g. $BV$ conditions for $\nabla u$) and flexibility of low regularity solutions arises \cite{MS,MS1,K,MSyl,DM96,DaM12, RZZ20, RZZ18,RTZ19,DPR20}.

Moreover, building on the first (more qualitative) step of investigating exactly stress-free structures, further quantitative properties of the resulting material patterns and the associated energies are studied in the literature. For instance, this includes the scaling and relaxation behaviour of the associated energies \cite{KM1, KM, Lorent06,Lorent09, CC15,KW14,KK11,KKO13,R16, RT22,RT21,RT22a, TS} as well as the stability and fine-scale properties of these patterns \cite{CO1,CO,C1,TS1}. We refer to \cite{M1} and \cite{B} for a survey of these results.

\subsection{Outline of the article}
The remainder of the article is structured as follows: In Section \ref{sec:linear_cond} we first exploit the linear structure conditions which are given by the Saint-Venant compatibility conditions and by the assumption of periodicity. Next, in Section \ref{sec:refined} we combine these with the nonlinear and non-convex constraints which arise from our differential inclusion and prove the crucial symmetry breaking in the form that the strains must have only two- instead of fully three-dimensional dependences. Last but not least, in Section \ref{sec:proof} we provide the proof of Theorem \ref{thm:structure}.

\section{First Structure Results: Exploiting the Saint-Venant Conditions}
\label{sec:linear_cond}

In this section we employ the Saint-Venant compatibility conditions to deduce first structure results for $e(u)$. Relying on these structural results, in the next section we will carry out a refined analysis of the restrictions of the strains to certain planes in order to prove that $e(u)$ only depends on two of the three variables.

We begin by recalling the Saint-Venant compatibility conditions (see for instance \cite[Lemma 1]{R16}):

\begin{lem}[Saint-Venant compatibility conditions]
Let $\Omega \subset \R^3$ be a simply connected, bounded domain and let $e:\Omega \rightarrow \R^{3\times 3}_{sym}$ be bounded. Then the following conditions are equivalent:
\begin{itemize}
\item[(i)] there exists a deformation $u \in W^{1,p}(\Omega)$ with $p\in (1,\infty)$ such that $e = \frac{1}{2}(\nabla u + (\nabla u)^t)$, i.e. $e$ is a strain corresponding to a deformation $u$,
\item[(ii)] distributionally it holds that $\nabla \times (\nabla \times e) =0$, i.e the following system of PDEs hold distributionally in $\Omega$
\begin{align}
\label{eq:wave_type}
\begin{split}
2 \p_{12} e_{12} & = \p_{22} e_{11} + \p_{11} e_{22},\\
2 \p_{13} e_{13} & = \p_{33} e_{11} + \p_{11} e_{33},\\
2 \p_{23} e_{23} &= \p_{22} e_{33} + \p_{33} e_{22}, 
\end{split}
\end{align}
\begin{align}
\label{eq:curl_type}
\begin{split}
\p_{23} e_{11} & = \p_1 (-\p_1 e_{23} + \p_2 e_{13} + \p_3 e_{12}),\\
\p_{13} e_{22} & = \p_2(\p_1 e_{23} - \p_2 e_{13} + \p_3 e_{12}),\\
\p_{12} e_{33} & = \p_3(\p_1 e_{23} + \p_{2} e_{13} - \p_3 e_{12}).
\end{split}
\end{align}
\end{itemize}
\end{lem}

In the following sections we will apply these compatibility equations to solutions to our differential inclusion \eqref{eq:diff_incl}. To this end, we note that by virtue of the constant diagonal entries of the wells from \eqref{eq:strains} the first three strain equations \eqref{eq:wave_type} can be simplified to read
\begin{align*}
\p_{12} e_{12} &= 0,\\
\p_{13} e_{13} &= 0,\\
\p_{23} e_{23} & = 0.
\end{align*}

Integrating this, we directly obtain the following decomposition into two-dimensional waves:
\begin{align}
	e_{12}(x_1,x_2,x_3) & = f_{31}(x_1,x_3) + f_{32}(x_2,x_3), \label{decomposition_12}\\
	e_{23}(x_1,x_2,x_3) &= f_{12}(x_1,x_2) + f_{13}(x_1,x_3),\label{decomposition_23}\\
	e_{31}(x_1,x_2,x_3) & = f_{21}(x_1,x_2) + f_{23}(x_2,x_3),\label{decomposition_31}
\end{align}
for functions $f_{ij} : \T_i\times \T_j \to \R$ (c.f. the argument for Lemma \ref{lem:1D_functions}(4) below).

The two-valuedness of the components of the strain allows to extract further information about these functions.

\begin{lem}\label{lem:1D_functions}
	Let $e(u): \T^3 \rightarrow \R^{3\times 3}_{sym}$ be a $\T^3$-periodic symmetrized gradient solving the differential inclusion \eqref{eq:diff_incl}.
	Then one can decompose the strain components as stated in \eqref{decomposition_12}-\eqref{decomposition_31} and choose the functions $f_{ij}$ for $i,j\in \{1,2,3\}$ with $i\neq j$ to satisfy the following three statements for $\{i,j,k\} = \{1,2,3\}$:
	\begin{enumerate}
		\item The functions $f_{ij}$ satisfy the inclusion $f_{ij} \in \{-1,0,1\}$ a.e. for $i\neq j$.
		\item For almost all $x_i \in \T_i$ we have either $f_{ij}(x_i,\bullet) = 0$ for almost all $x_j \in \T_j$ or $f_{ik}(x_i,\bullet) = 0$ for almost all $x_k \in \T_k$.
		\item If there exists $C \in \R$ such that $e_{ij} = C$, then $f_{ki}$ and $f_{kj}$ can be chosen to be constant.
		\item The functions $f_{ij}$ are periodic in both variables, i.e. for almost every  $(x_i, x_j) \in \T_i \times \T_j $ it holds that $f_{ij}(x_i+\lambda_i,x_j)-f_{ij}(x_i,x_j)=0$ and $f_{ij}(x_i, x_j+ \lambda_j)- f_{ij}(x_i, x_j) = 0$.
	\end{enumerate}
\end{lem}

\begin{proof}
The claims follow immediately as by the two-valuedness of the strain components, it is possible to rewrite \eqref{decomposition_12}-\eqref{decomposition_31} as 
\begin{align}
\label{eq:strain_char}
\begin{split}
e_{12}(x_1,x_2,x_3)&= \chi_3(x_3)(1-2\chi_{31}(x_1,x_3)) + (1-\chi_3(x_3))(1-2\chi_{32}(x_2,x_3)),\\
e_{23}(x_1,x_2,x_3)&= \chi_1(x_1)(1-2\chi_{12}(x_1, x_2)) + (1-\chi_1(x_1))(1-2\chi_{13}(x_1, x_3)),\\
e_{31}(x_1,x_2,x_3)&= \chi_2(x_2)(1-2\chi_{21}(x_1,x_2)) + (1-\chi_2(x_2))(1-2\chi_{23}(x_2,x_3)),
\end{split}
\end{align}
where $\chi_j, \chi_{ik}$ are characteristic functions, which can be chosen to be constant if the corresponding entry of the strain $e_{ij}$ is constant.

Finally, we turn to the periodicity claim. By symmetry it suffices to consider $f_{12}$. We note that by the decomposition of the strain into two planar waves, we immediately obtain that for almost every $(x_1,x_2) \in \T_1 \times \T_2$
\begin{align*}
0 = e_{23}(x_1,x_2+\lambda_2,x_3)- e_{23}(x_1,x_2,x_3) = f_{12}(x_1,x_2+\lambda_2)- f_{12}(x_1,x_2).
\end{align*}
In order to also infer the periodicity in the $x_1$ variable, we use that for almost every $(x_1,x_2) \in \T_1 \times \T_2$
\begin{align}
\begin{split}
\label{eq:strain_period}
0 &= e_{23}(x_1+\lambda_1,x_2,x_3)- e_{23}(x_1,x_2,x_3) \\
&= f_{12}(x_1+\lambda_1,x_2)- f_{12}(x_1, x_2) + f_{13}(x_1+\lambda_1,x_3)- f_{13}(x_1, x_3).
\end{split}
\end{align}
Varying the $x_2,x_3$ variables separately, we deduce that for almost every $(x_1,x_2) \in \T_1 \times \T_2$ it holds that $f_{12}(x_1+\lambda_1,x_2)- f_{12}(x_1, x_2) = const.$ Due to the structural result from \eqref{eq:strain_char} this is only possible if $\chi_{1}(x_1+\lambda_1) = \chi_{1}(x_1)$ for almost every $x_1 \in \T_1$. For all $x_1 \in \T_1$ with $\chi_{1}(x_1+\lambda_1) = 0$, the claim of the lemma follows immediately. For $x_1 \in \T_1$ with $\chi_{1}(x_1+\lambda_1) =1$, the periodicity condition \eqref{eq:strain_period} of the strain and the fact that for these choices of $x_1 \in \T_1$ it necessarily holds that $f_{13}(x_1+\lambda_1,x_3) = f_{13}(x_1,x_3) = 0$, then also yields the desired periodicity condition for $f_{12}$. 
\end{proof}

Next we consider the second set of strain equations \eqref{eq:curl_type} which simplify to become
\begin{align}
0&= \p_1(-\p_1 e_{23} + \p_2 e_{13} + \p_3 e_{12}), \label{strain_equation_2_1}\\
0&= \p_2(\p_1 e_{23} - \p_2 e_{13} + \p_3 e_{12}), \label{strain_equation_2_2}\\
0&= \p_3(\p_1 e_{23} + \p_2 e_{13} - \p_3 e_{12}). \label{strain_equation_2_3}
\end{align}
By integrating these, we infer a further structure result:

\begin{lem}\label{lem:constant_averages}
Let $e(u): \T^3 \rightarrow \R^{3\times 3}_{sym}$ be a $\T^3$-periodic symmetrized gradient solving the differential inclusion \eqref{eq:diff_incl}.
Then, for all ${i,j,k} =\{1,2,3\}$ we have that
\begin{align}\label{averages_constant}
	\dashint_{\T_i\times\T_j} e_{ij}(x_i,x_j,x_k) \intd (x_i,x_j) = \dashint_{\T^3} e_{ij}(x) \intd x \text{ for almost every } x_k \in \T_k.
\end{align}
\end{lem}

\begin{proof}
By symmetry, we only have to consider the case $i=1$, $j=2$, and $k=3$. 
From equation \eqref{strain_equation_2_3}, we obtain
\begin{align*}
\p_1 e_{23} + \p_2 e_{13} - \p_3 e_{12} = h(x_1,x_2).
\end{align*}
Integrating in $x_1,x_2$ and using the periodicity assumption then yields that for almost every $x_3 \in \T_3$
\begin{align*}
\p_3 \dashint_{\T_1 \times \T_2} e_{12}(x_1,x_2,x_3) \intd x_1 \intd x_2
&= \dashint_{\T_1 \times \T_2} h(x_1,x_2) \intd x_1 \intd x_2\\
& \quad - \dashint_{\T_1 \times \T_2} \p_1 e_{23} \intd x_1 \intd x_2
- \dashint_{\T_1 \times \T_2} \p_2 e_{12} \intd x_1 \intd x_2\\
&=  \dashint_{\T_1 \times \T_2} h(x_1,x_2) \intd x_1 \intd x_2.
\end{align*}
Now, integrating this in $x_3$ and using the periodicity of the strain component $e_{12}$, implies that 
\begin{align*}
0 = \dashint_{\T^3} h(x_1,x_2) \intd x_1 \intd x_2 \intd x_3 = \dashint_{\T_1 \times \T_2} h(x_1,x_2) \intd x_1 \intd x_2.
\end{align*}
Hence, by the above computation, $\p_3 \dashint_{\T_1 \times \T_2} e_{12}(x_1,x_2,x_3) \intd x_1 \intd x_2 = 0$, which concludes the argument.
\end{proof}

By combining the information from both sets of strain equations, we can prove the existence of a periodic primitive which in turn is closely related to the planar waves from Lemma \ref{lem:1D_functions}.

\begin{lem}\label{lem:primitive}
Let $e(u): \T^3 \rightarrow \R^{3\times 3}_{sym}$ be a $\T^3$-periodic symmetrized gradient solving the differential inclusion \eqref{eq:diff_incl} and let $f_{ij}$ with $i,j\in\{1,2,3\}$, $i\neq j$, denote the functions from Lemma \ref{lem:1D_functions}.
	Then there exist periodic Lipschitz vector fields $\Psi_{i,i+1} : \T_i \times \T_{i+1} \to \R$ for the cyclical indices $i  \in \{1,2,3\}$ such that with $\Psi_{i+1,i}:= \Psi_{i,i+1} $ the following properties hold:
	\begin{enumerate}
		\item We have the decomposition
			\begin{align*}
				e_{12} & =  \partial_2 \Psi_{23} + \partial_1 \Psi_{31} \phantom{ {}+ \partial_2 \Psi_{12}} + \dashint e_{12} \intd x,\\
				e_{23} & = \phantom{ \partial_3 \Psi_{23} +{}} \partial_3 \Psi_{31}  + \partial_2 \Psi_{12} + \dashint e_{23} \intd x,\\
				e_{31} & = \partial_3 \Psi_{23} \phantom{{}+ \partial_3 \Psi_{31}}  + \partial_1 \Psi_{12} + \dashint e_{31} \intd x.
			\end{align*}
		\item The primitives satisfy the discrete differential inclusion
			\[\partial_j \Psi_{ij} \in \left\{-1 - \dashint_{\T^3} e_{jk} \intd x,1 - \dashint_{\T^3} e_{jk} \intd x,0\right\}\]
			for $\{i,j,k\}=\{1,2,3\}$.
		\item They allow to efficiently detect whether $f_{ij} \equiv const$ in some direction for $i,j\in \{1,2,3\}$ with $i\neq j$: Let $x_i \in \T_i$ be fixed such that $f_{ij}(x_i, \bullet)$ and $\p_j \Psi_{ij}(x_i , \bullet)$ are measurable functions. Then the following properties are equivalent:
			\begin{itemize}
				\item[(i)] $f_{ij}(x_i,\bullet) \equiv  const.$
				\item[(ii)] $\partial_j \Psi_{ij}(x_i, \bullet ) \equiv 0.$
				\item[(iii)] $\left|\left\{x_j \in \T_j: \partial_j \Psi_{ij}(x_i, x_j) = 0 \right\}\right| > 0$.
			\end{itemize}
	\end{enumerate}
\end{lem}

\begin{proof}
We argue in three steps, first constructing the primitive which yields the identity stated in (1), then derive the properties in (2) and finally deduce the equivalences in (3).\\

\textit{Step 1: Construction of the primitive.}\\
In order to deduce the desired structure result, we combine the strain equations \eqref{strain_equation_2_1}-\eqref{strain_equation_2_3} with the representation formulae from Lemma \ref{lem:1D_functions}.

We take the $\p_2$ derivative of the first strain equation \eqref{strain_equation_2_1} and subtract the $\p_1$ derivative of the second equation \eqref{strain_equation_2_2} to infer the first equation in 

	\begin{align}
	\label{eq:derivative_syst}
	\begin{split}
		0&= -\p_1^2 \p_2 e_{23} + \p_1\p_2^2 e_{13} ,\\
		0&= - \p_2^2\p_3 e_{13} + \p_2\p_3^2 e_{12},\\
		0&= +\p_1^2 \p_3 e_{23} -\p_1\p_3^2 e_{12},\\
		\end{split}
	\end{align}
	while the others follow by symmetry.
	
	Using the representation from Lemma \ref{lem:1D_functions} and evaluating the first equation from \eqref{eq:derivative_syst} yields
	\begin{align*}
		\p_1^2 \p_2 f_{12}(x_1,x_2) = h_{12}(x_1,x_2)=\p_{1}\p_{2}^2 f_{21}(x_1,x_2).
	\end{align*}
	Here $h_{12}(x_1,x_2)$ denotes a generic function in the $x_1,x_2$ variables which may change from each block of equations to the next, as do the functions $g_{12}$, $g_{21}$, $k_{12}$, and $k_{21}$ from what follows for the respective arguments.
	Integrating in the $x_1$ direction hence results in
	\begin{align*}
		\p_{1}\p_2 f_{12}(x_1,x_2) &= h_{12}(x_1,x_2) + g_{12}(x_2),\\
		\p_{2}^2 f_{21}(x_1,x_2) &= h_{12}(x_1, x_2) + g_{21}(x_2).
	\end{align*}
	Integrating in the $x_2$ direction then gives
	\begin{align*}
		\p_{1} f_{12}(x_1,x_2) &= h_{12}(x_1,x_2) + g_{12}(x_2) + k_{12}(x_1),\\
		\p_{2} f_{21}(x_1,x_2) &= h_{12}(x_1, x_2) + g_{21}(x_2) + k_{21}(x_1).
	\end{align*}
In particular, for functions $\bar{g}_{12}: \T_2 \to \R$ and $\bar{k}_{12}: \T_1 \to \R$ this yields
\begin{align*}
\p_1 f_{12}(x_1,x_2) = \p_2 f_{21}(x_1,x_2) - \bar{g}_{12}(x_2)+ \bar{k}_{12}(x_1).
\end{align*}	
Defining $\overline{h}_{12}(x_1,x_2):= \p_2 f_{21}(x_1,x_2) - \bar{g}_{12}(x_2)$, we then infer
\begin{align*}
\p_1 f_{12}(x_1,x_2)&= \overline{h}_{12}(x_1,x_2) + \bar{k}_{12}(x_1),\\
\p_2 f_{21}(x_1,x_2)&= \overline{h}_{12}(x_1,x_2) + \bar{g}_{12}(x_2).
\end{align*}
For convenience of notation, we drop the bars in the sequel and simply write 
	\begin{align}
		\p_{1} f_{12}(x_1,x_2) &= h_{12}(x_1,x_2) + k_{12}(x_1),\label{rot_free_1}\\
		\p_{2} f_{21}(x_1,x_2) &= h_{12}(x_1, x_2) + g_{12}(x_2).\label{rot_free_2}
	\end{align}
We next seek to prove that the vector field $\begin{pmatrix} f_{21} \\ f_{12} \end{pmatrix}$ from equations \eqref{rot_free_1} and \eqref{rot_free_2} essentially comes from a gradient field in the two variables $x_1,x_2$.	

	Averaging over $\T_1\times \T_2$ in the equations \eqref{rot_free_1}, \eqref{rot_free_2} and recalling the periodicity of the functions $f_{ij}$ implies that \[\dashint k_{12} \intd x_1 = \dashint g_{12} \intd x_2.\]
	Thus, possibly after modifying $h_{12}$ by a constant, we can choose \[\dashint k_{12} \intd x_1 = \dashint g_{12} \intd x_2 = 0.\]
Returning to \eqref{rot_free_1}, \eqref{rot_free_2} and invoking the periodicity of $f_{12}$, $f_{21}$, we have for each $\bar{x}_1 \in \T_1, \bar{x}_2 \in \T_2$ that
	\[\dashint h_{12}(x_1,\bar{x}_2) \intd x_1 =  \dashint h_{12}(\bar{x}_1, x_2) \intd x_2 = 0.\]
	This in turn implies
	\begin{align*}
		 k_{12}(x_1) & = \dashint \p _1 f_{12}(x_1,x_2) \intd x_2,\\
		 g_{12}(x_2) & = \dashint \p _2 f_{21}(x_1,x_2) \intd x_1.
	\end{align*}
	For $i,j\in \{1,2,3\}$ with $i\neq j$ setting 
	\begin{align*}
		  G_{ij}(x_i) & := \dashint  f_{ij}(x_i,x_j) \intd x_j,
	\end{align*}
	we see that equations \eqref{rot_free_1} and \eqref{rot_free_2}
	turn into
		\begin{align*}
		\p_{1} (f_{12}(x_1,x_2) - G_{12}(x_1)) &= h_{12}(x_1,x_2) ,\\
		\p_{2} (f_{21}(x_1,x_2)-G_{21}(x_2)) &= h_{12}(x_1, x_2).
	\end{align*}
	Thus, we deduce the existence of a periodic primitive $\Psi_{12} : \T_1 \times \T_2 \to \R$ with
	\begin{align}
	  \partial_1 \Psi_{12} & = f_{21} - G_{21},\\
	  \partial_2 \Psi_{12}  & = f_{12} - G_{12}.\label{def_primitive_12}
	\end{align}
	We note that the periodicity of $\Psi_{12}$ follows from the fundamental theorem of calculus and the mean zero conditions for $\p_1 \Psi_{12}$ and for $\p_2 \Psi_{12}$.
	By cyclical symmetry, we also deduce the existence of $\Psi_{23}$ and $\Psi_{31}$ such that
	\begin{align}\label{definition_primitive}
	\begin{split}
		\partial_2 \Psi_{23} & = f_{32} - G_{32},\\
		\partial_3 \Psi_{23} & = f_{23} - G_{23},\\
		\partial_3 \Psi_{31} & = f_{13} - G_{13},\\
		\partial_1 \Psi_{31} & = f_{31} - G_{31}.
		\end{split}	
	\end{align}
	With this in hand, using the representation \eqref{decomposition_12}-\eqref{decomposition_31}, we first obtain
	\begin{align}
	\label{eq:struc_strain}
	\begin{split}
	e_{12}(x_1, x_2, x_3) &= f_{31}(x_1,x_2) + f_{32}(x_2,x_3)\\
	&= \p_1 \Psi_{31}(x_1,x_3) + G_{31}(x_3) + \p_{2}\Psi_{23}(x_2,x_3) + G_{32}(x_3).
	\end{split}
	\end{align}
	Then, recalling the periodicity of $\Psi_{ij}$ and applying Lemma \ref{lem:constant_averages} implies that 
	\begin{align*}
	G_{31}(x_3) + G_{32}(x_3) = \dashint e_{12} \intd x.
	\end{align*}
	Combining this with \eqref{eq:struc_strain} then concludes the proof of
	the first statement of the Proposition up to cyclical symmetry.\\

\emph{Step 2: Proof of (2).} In order to observe (2), for simplicity, we consider the case $i=1$, $j=2$ and $k=3$, i.e. we seek to prove that
\begin{align}
\label{eq:three_val}
\p_2 \Psi_{12} \in \left\{ -1 - \dashint_{\T^3} e_{23} \intd x , 1-\dashint_{\T^3} e_{23} \intd x,0 \right\}.
\end{align}
We deduce the differential inclusion \eqref{eq:three_val} as a consequence of the definition of $\Psi_{12}$ in terms of the functions $f_{12}$ and $f_{21}$. Indeed, 
\begin{align}
\label{eq:grad_expr}
\begin{pmatrix} \p_1 \Psi_{12}(x_1,x_2)\\ \p_2 \Psi_{12}(x_1,x_2) \end{pmatrix}
= \begin{pmatrix} f_{21}(x_1,x_2)-\dashint_{\T_1} f_{21}(x_1,x_2) \intd x_1 \\ f_{12}(x_1,x_2) - \dashint_{\T_2} f_{12}(x_1,x_2) \intd x_2 \end{pmatrix}.
\end{align}
Now, using Lemma \ref{lem:1D_functions}, we have that $f_{12}\in \{0,1,-1\}$ and that $e_{23} = f_{12} + f_{13}$.
In the following we choose $\bar{x}_1\in \T_1$ such that all functions are still measurable as restrictions to $\{\bar{x}_1\} \times \T_2$, $\{\bar{x}_1\} \times \T_3$ or $\{\bar{x}_1\} \times \T_2 \times \T_3$, which is the case for almost all $\bar{x}_1 \in \T_1$.

If we have $f_{12}(\bar{x}_1,x_2) =0$ for a set of positive measure in $x_2$, then by Lemma \ref{lem:1D_functions} we obtain that $f_{12}(\bar{x}_1,x_2) =0$ for almost all $x_2\in \T_2$.
Thus, by construction \eqref{eq:grad_expr}, also $\p_2 \Psi_{12}(\bar{x}_1,x_2) =0$ for almost all $x_2 \in \T_2$.

If however we have $|f_{12}(\bar{x}_1,x_2)| = 1$ for some set of positive measure in $x_2$, then by Lemma \ref{lem:1D_functions} we have $|f_{12}(\bar{x}_1,x_2)|=1$ for almost all $x_2 \in \T_2$ and
 $f_{13}(\bar{x}_1,x_3)=0$ for almost all $x_3 \in \T_3$.
 As a consequence, 
\begin{align*}
\dashint_{\T_2} f_{12}(\bar{x}_1, x_2) \intd x_2 = \dashint_{\T_2 \times \T_3} e_{23}(\bar{x}_1, x_2, x_3) \intd x_2 \intd x_3 = \dashint_{\T^3} e_{23} \intd x.
\end{align*}
In the last equality we used Lemma \ref{lem:constant_averages}.
Due to $|f_{12}(\bar{x}_1,x_2)|=1$ for almost all $x_2$ and due to the representation \eqref{eq:grad_expr}, we finally obtain the differential inclusion \eqref{eq:three_val}.

This concludes the argument for (2). \\

\textit{Step 3: Proof of (3).} The implications $(i)\Rightarrow (ii)$ and $(ii)\Rightarrow (iii)$ follow directly from the definitions of the functions $\p_j \Psi_{ij}$. In order to infer the equivalences, it thus suffices to prove that (iii) implies (i). For simplicity of notation, we assume that $i=1$, $j=2$, $k=3$.

Recall that for the fixed $x_1$ all restrictions are measurable.
Assuming that (iii) holds, we obtain that there exists a set $E\subset \T_2$ with positive one-dimensional Lebesgue measure such that for $x_2 \in E$ we have $\p_{2} \Psi_{12}(x_1,x_2)=0$. As a consequence, by construction of $\p_2 \Psi_{12}$  for $x_2 \in E$ we obtain
\begin{align}
\label{eq:psi_zero}
f_{12}(x_1,x_2) = \dashint\limits_{\T_2} f_{12}(x_1, x_2) d x_2 \in [-1,1].
\end{align} 
Since by Lemma \ref{lem:1D_functions} we also have $f_{12}\in \{0,\pm 1\}$, the identity \eqref{eq:psi_zero} yields that also
\begin{align*}
	\dashint\limits_{\T_2} f_{12}(x_1, x_2) d x_2 \in  \{\pm 1,0\}.
\end{align*}
If $\dashint\limits_{\T_2} f_{12}(x_1, x_2) d x_2 \in  \{\pm 1\}$, this immediately implies that by discreteness and extremality of the values $\pm 1$ it holds $f_{12}(x_1, \bullet) \equiv \pm 1$ which proves the claim.

It hence suffices to consider the case that $\dashint\limits_{\T_2} f_{12}(x_1, x_2) d x_2 = 0 $. 
Again by \eqref{eq:psi_zero} this however implies that $f_{12}(x_1,\bullet ) = 0$ on a set of positive measure. By Lemma \ref{lem:1D_functions}(2) this in turn results in $f_{12}(x_1,\bullet) = 0$ for almost all $x_2 \in \T_2$ which also proves the claim (3) in the last remaining case.
\end{proof}

\section{A Refined Analysis of the Functions $\Psi_{jk}$}

\label{sec:refined}

In this section, we use the structure results which were deduced from the Saint-Venant conditions and in particular the conditions from Lemma \ref{lem:primitive_constant}(3) in order to obtain finer information on the potentials $\Psi_{jk}$. Here we exploit a final, nonlinear property of the strains from \eqref{eq:strains}, namely
\begin{align*}
e_{12} e_{23} - e_{13} = 0.
\end{align*}
We will use this property (and permutations thereof) on carefully chosen planes in order to obtain further conditions on the potentials $\Psi_{ij}$.
Note that the following result is \emph{not} symmetric in the indices anymore. Indeed, it is in the following lemma that the symmetry is broken for the first time which then subsequently entails the desired one-dimensional dependence of the strain components.

\begin{lem}\label{lem:primitive_constant}
Let $e(u): \T^3 \rightarrow \R^{3\times 3}_{sym}$ be a $\T^3$-periodic symmetrized gradient solving the differential inclusion \eqref{eq:diff_incl} and let $f_{ij}$ with $i,j\in\{1,2,3\}$, $i\neq j$, denote the functions from Lemma \ref{lem:1D_functions}.
	Let $\left|\left\{x_i \in \T_i: f_{ij}(x_i,\bullet) \equiv const \text{ a.e.}\right\} \right| > 0$ for some $\{i,j,k\} \in \{1,2,3\}$.
	Then at least one of the following results holds: Almost everywhere we have 	
	\begin{align*}
	\Psi_{ij} \equiv const \mbox{ or }  \Psi_{jk} \equiv \Psi_{jk}(x_j) \mbox{ or } \Psi_{ik} \equiv \Psi_{ik}(x_k).
	\end{align*}
\end{lem}

\begin{proof}
	For simplicity, we choose $i=3$ and $j=2$. Thus, we seek to prove that at least one of the following structure results holds $\Psi_{32}\equiv const$ or $\Psi_{21} \equiv \Psi_{21}(x_2)$ or $\Psi_{31} \equiv \Psi_{31}(x_1)$.
	The other cases follow by symmetry.

	Let $A \subset \T_3$ be the set such that for $x_3 \in A$ we have $f_{32}(x_3,\bullet) \equiv const$ and such that all involved functions are defined $\mathcal{L}^2$-a.e.\ on $\T_1\times \T_2 \times \{x_3\}$.
	Note that by assumption we have
	\begin{align}
		|A| > 0. \label{non-empty}
	\end{align}
	In the following, we will fix $x_3 \in A$ and consider all functions to be restricted to this hyperplane by abuse of notation.
	In particular, on any such plane we have $e_{12} = e_{12}(x_1)$.\\

\emph{Step 1: Reduction by means of a transport equation.}	
	The identity $e_{31} - e_{12}e_{23} = 0$ together with the decomposition in Lemma \ref{lem:primitive} implies
	\begin{align*}
		\partial_1 \Psi_{12} - e_{12} \partial_2 \Psi_{12} = - \partial_3 \Psi_{23} + e_{12}\partial_3 \Psi_{31} - \dashint_{\T^3} e_{31} \intd x + e_{12} \dashint_{\T^3} e_{23} \intd x.
	\end{align*}
	Defining $E_{12}(x_1)$ to satisfy
	\begin{align*}
		E_{12}'(x_1) = e_{12}(x_1) \text{ and } E_{12}(0) = 0,
	\end{align*}
	we infer that almost everywhere 
	\begin{align*}
		\partial_1 \left( \Psi_{12} \left( x_1, x_2 - E_{12}(x_1) \right) \right) = - (\partial_3 \Psi_{23})(x_2 - E_{12}(x_1) )  + h(x_1),
	\end{align*}
	where $h$ is a generic measurable and bounded function of $x_1$ that may change from line to line.
	Here the use of the chain rule for the Lipschitz-function $\Psi_{12}$ is justified as the transformation $(x_1,x_2) \mapsto (x_1,x_2 - E_{12}(x_1))$ is volume-preserving.
	
Upon integrating in $x_1$ we get
	\begin{align*}
		\Psi_{12}\left(x_1,x_2-E_{12}(x_1)\right) = -\int_{0}^{x_1}\partial_3 \Psi_{23}(x_2 - E_{12}(s) )\intd s + h(x_1) + k(x_2)
	\end{align*}
	almost everywhere, where $k$ is a measurable function of $x_2$.
	By taking the difference of the above equation for $x_1, \tilde x_1\in \T_1$ we obtain
	\begin{align}\label{replacement_affine}
		\begin{split}
			& \quad \Psi_{12} \left(x_1,x_2 - E_{12}(x_1) \right) - \Psi_{12} \left(\tilde x_1,x_2 - E_{12}( \tilde x_1) \right) - h(x_1) + h(\tilde x_1) \\
			& =  -  \int_{\tilde x_1}^{x_1}\partial_3 \Psi_{23}(x_2 - E_{12}(s) )\intd s.
		\end{split}
	\end{align}

\emph{Step 2: Consequences of the transport equation.}
Let $B:=\{x_1 \in \T_1 : f_{12}(x_1,\bullet) \equiv const\} =\{x_1 \in \T_1 : \p_2 \Psi_{12}(x_1,\bullet) \equiv 0\} $, where the equality of the two sets follows from Lemma \ref{lem:primitive}(3).
We now distinguish two cases:\\
\emph{Step 2.1:}
Firstly, we assume that $\left| B\right| = 0$. Then, $f_{12}(x_1, \bullet) \not \equiv const$ for almost all $x_1 \in \T_1$, which by Lemma \ref{lem:1D_functions} implies $f_{13} \equiv 0$.
	In turn, we get $\Psi_{31} \equiv \Psi_{31}(x_1)$ by Lemma \ref{lem:primitive}.\\

\emph{Step 2.2:}
	Let us consider the case $|B| >0$.
	As for $x_1, \tilde x_1 \in B$, the left hand side of identity \eqref{replacement_affine} is independent of $x_2$, there exists a Lipschitz continuous function $K : B \to \R$ such that for almost all  $(x_1, \tilde x_1) \in B^2$ we have
	\begin{align}\label{difference_Lipschitz}
		K(x_1)-K(\tilde x_1) =  \int_{\tilde x_1}^{x_1}\partial_3 \Psi_{23}(x_2 - E_{12}(s) )\intd s.
	\end{align}
	By Kirszbraun's theorem, we may consider $K$ to be defined on $\T_1$ and thus it is differentiable almost everywhere by Rademacher's theorem.
	
	Let $x_2 \in \T_2$.
	Then the map $ s \mapsto \partial_3 \Psi_{23}(x_2 - E_{12}(s))$ is measurable.
	Therefore, by the Lebesgue point theorem, we have for almost all $x_1 \in \T_1$ that
	\begin{align}\label{Lebesgue point}
		\lim_{\eps \to 0} \dashint_{x_1- \eps}^{x_1 + \eps} \partial_3 \Psi_{23}(x_2 - E_{12}(s) )\intd s = \partial_3 \Psi_{23}(x_2 - E_{12}(x_1) ).
	\end{align}
	By a reflection argument, for all $x_1 \in B$ of density one there exists $\eps_n >0$ for $n\in \N$, depending on $x_1$, such that $\eps_n \to 0$ and $x_1 \pm \eps_n \in B$.
	The identity \eqref{difference_Lipschitz} implies that for $x_1$ satisfying the above requirements, we have
	\begin{align*}
		\dashint_{x_1- \eps_n}^{x_1 + \eps_n} \partial_3 \Psi_{23}(x_2 - E_{12}(s) )\intd s = \frac{1}{2\eps_n} (K(x_1+\eps_n) - K(x_1 -\eps_n)). 
	\end{align*}
	Therefore, the convergence \eqref{Lebesgue point} gives
	\begin{align}
	\label{eq:diff_Leb}
		\partial_3 \Psi_{23}(x_2 - E_{12}(x_1) ) =  K'(x_1)
	\end{align}
	for all $x_2 \in \T_2$ and almost all $x_1 \in B$.

	Consequently, varying $x_2$ in \eqref{eq:diff_Leb}, we observe that $\partial_3 \Psi_{23}(\bullet, x_3) \equiv const$ for $x_3 \in A$, which implies
	$\partial_3 \Psi_{23} (x_2,x_3) = q(x_3)$
	for almost all $(x_2,x_3)\in \T_2 \times A$.
We next claim that the fact that $\partial_3 \Psi_{23} (x_2,x_3) = q(x_3)$
	for almost all $x_2 \in \T_2$ and $x_3 \in A$ together 	
with Lemma \ref{lem:primitive}(3) implies the dichotomy
	\begin{align}
	\label{eq:dichotomy}
	\begin{split}
		f_{23}(x_2,\bullet) \equiv const \text{ for a.e.\ } x_2 \in \T_2 \text{ or } f_{23} (x_2, \bullet) \not \equiv const \text{ for a.e.\ } x_2 \in \T_2.
		\end{split}
	\end{align}
	Indeed, we distinguish two cases:
	\begin{itemize}
	\item[(i)] If for some set of positive measure in $x_2\in \T_2$ and for
	\begin{align*}
	C(x_2):=\{x_3 \in \T_3: \ \p_3 \Psi_{23}(x_2,x_3)=0\ \}
	\end{align*}
	we have $|C(x_2)|>0$,
	then by Lemma \ref{lem:primitive}(3) we obtain that for such a value of $x_2\in \T_2$  and almost all $x_3 \in \T_3$ it holds $\p_3 \Psi_{23}(x_2,x_3)\equiv 0$.
	By virtue of the fact that $\p_3 \Psi_{23}(x_2, x_3) = q(x_3)$ for $(x_2,x_3)\in \T_2 \times A$, we then obtain that for all $(x_2,x_3)\in \T_2 \times A$ it holds that $q(x_3)= 0$ and that $\p_3 \Psi_{23}(x_2, x_3) =0$. 
	 Hence, the condition $|C(x_2)|>0$ holds for almost all $x_2 \in \T_2$. But then Lemma \ref{lem:primitive}(3) implies that $f_{23}(x_2,\bullet) \equiv const$ holds for almost all $x_2 \in \T_2$. 
	\item[(ii)] If for almost all $x_2 \in \T_2$ it holds that
	\begin{align}
	\label{eq:meas1}
	|C(x_2)|=|\{x_3 \in \T_3: \ \p_3 \Psi_{23}(x_2,x_3)=0\}|=0,
	\end{align}
by Lemma \ref{lem:primitive}(3) we have for all $x_2\in \T_2$ that $f_{23}(x_2,\bullet ) \not \equiv const$.
	\end{itemize}

	Therefore, from (i), (ii) we conclude \eqref{eq:dichotomy}. If combined with Lemma \ref{lem:1D_functions}, this in turn yields
		\begin{align}
		\label{eq:dichot_1}
		f_{23}(x_2,\bullet) \equiv const \text{ for a.e.\ } x_2 \in \T_2  \text{ or } f_{21} \equiv 0 \text{ a.e. } \text{ on } \T_1\times \T_2.
	\end{align}
	
	In the first case of \eqref{eq:dichot_1} by Lemma \ref{lem:primitive}(3) we have $\Psi_{23} \equiv \Psi_{23}(x_2)$.
	Moreover, by Lemma \ref{lem:primitive}(3), the assumption \eqref{non-empty} that there exists $x_3 \in \T_3$ such that $f_{32}(x_3,\bullet) \equiv const$, gives that there exist $x_3 \in \T_3$ such that $\partial_2 \Psi_{32}(x_3,\bullet) \equiv 0$.
	Hence, since $\Psi_{23} = \Psi_{32}$, we obtain that $\Psi_{32} \equiv const$ in this case.
	
	In the second case of \eqref{eq:dichot_1}, Lemma \ref{lem:primitive}(3) implies that $\Psi_{12} \equiv \Psi_{12}(x_2)$.
\end{proof}

As a consequence of the previous structure result, we show that we can dispose of one of the functions $\Psi_{ij}$ in the decomposition of the strain:

\begin{cor}\label{cor:primitive_constant}
Let $e(u): \T^3 \rightarrow \R^{3\times 3}_{sym}$ be a $\T^3$-periodic symmetrized gradient solving the differential inclusion \eqref{eq:diff_incl} and let $\Psi_{i,i+1}$ with $i\in\{1,2,3\}$ denote the functions from Lemma \ref{lem:primitive}.
	Then there exists an index $i\in \{1,2,3\}$ such that $\Psi_{i,i+1} \equiv const$.
\end{cor}

\begin{proof}
	By Lemma \ref{lem:1D_functions}(2) we have $\left|\{x_1: f_{12}(x_1,\bullet) \equiv const\} \right| >0$ or $\left|\{x_1: f_{13}(x_1,\bullet) \equiv const\} \right| >0$. We split the argument into two cases.
	
	\textit{Case 1: We assume that either
	\[\left|\{x_1 \in \T_1: f_{12}(x_1,\bullet) \equiv const\} \right| = |\T_1|\]
	or
	\[\left|\{x_1\in \T_1: f_{13}(x_1,\bullet) \equiv const\} \right| = |\T_1|.\]}
	Exchanging the indices 2 and 3 if necessary, it is sufficient to consider
	\[\left|\{x_1\in \T_1: f_{12}(x_1,\bullet) \equiv const\} \right| = |\T_1|,\]
	which implies $\partial_2 \Psi_{12}\equiv 0$.
	Lemma \ref{lem:primitive_constant} implies that at least one of the following statements holds: $\Psi_{12} \equiv const$ or $\Psi_{23} \equiv \Psi_{23}(x_2)$ or $\Psi_{31} \equiv \Psi_{31}(x_3)$.
	
	In the first case there is nothing left to prove.
	In the second case, using our assumption, the decomposition of Lemma \ref{lem:primitive} reads
	\begin{align}
	\label{eq:strain_decomp_aux}
	\begin{split}
				e_{12} & =  \partial_2 \Psi_{23}(x_2) + \partial_1 \Psi_{31}(x_3,x_1) \phantom{ {}+ \partial_2 \Psi_{12}(x_1)} + \dashint e_{12} \intd x,\\
				e_{23} & = \phantom{ \partial_3 \Psi_{23}(x_2) +{}} \partial_3 \Psi_{31}(x_3,x_1)  \phantom{{}+ \partial_2 \Psi_{12}(x_1)} + \dashint e_{23} \intd x,\\
				e_{31} & = \phantom{\partial_3 \Psi_{23}(x_2) + \partial_3 \Psi_{31}(x_3,x_1)  +{}} \partial_1 \Psi_{12}(x_1) + \dashint e_{31} \intd x.
	\end{split}
	\end{align}
Since the only function in the decomposition \eqref{eq:strain_decomp_aux} which depends on $x_2$ is given by $\p_2 \Psi_{23}(x_2)$ (which in turn only appears in the expression for $e_{12}$), the identity $e_{12} - e_{23}e_{31} \equiv 0$ trivially implies $\partial_2\Psi_{23} (\bullet) \equiv const$.
	By periodicity and the fundamental theorem of calculus, we get that the constant has to be zero and thus we obtain $\Psi_{23} \equiv const$.
	
In the third case, we have
	\begin{align*}
				e_{12} & =  \partial_2 \Psi_{23}(x_2,x_3) \phantom{ {}+ \partial_1 \Psi_{31}(x_3) + \partial_2 \Psi_{12}(x_1)} + \dashint e_{12} \intd x,\\
				e_{23} & = \phantom{ \partial_3 \Psi_{23}(x_2,x_3) +{}} \partial_3 \Psi_{31}(x_3)  \phantom{{}+ \partial_2 \Psi_{12}(x_1)} + \dashint e_{23} \intd x,\\
				e_{31} & = \partial_3 \Psi_{23}(x_2,x_3) \phantom{{}+ \partial_3 \Psi_{31}(x_3) } + \partial_1 \Psi_{12}(x_1) + \dashint e_{31} \intd x.
	\end{align*}
	Similarly as in the previous case we have $\partial_1 \Psi_{12} \equiv const$ and thus $\Psi_{12}= const$.
	
	\textit{Case 2: We assume that
	\[\left|\{x_1 \in \T_1: f_{12}(x_1,\bullet) \equiv const\} \right| >0\]
	and
	\[\left|\{x_1 \in \T_1: f_{13}(x_1,\bullet) \equiv const\} \right| >0.\]}
	Again, we only have to deal with the cases in which Lemma \ref{lem:primitive_constant} does not immediately give the desired statement.
	In that case we have 
	\begin{align*}
	\Psi_{23} \equiv \Psi_{23}(x_2) \mbox{ or }  \Psi_{31} \equiv \Psi_{31}(x_3)
	\end{align*}
	 and 
	 \begin{align*}
	 \Psi_{23} \equiv \Psi_{23}(x_3) \mbox{ or } \Psi_{12} \equiv \Psi_{12}(x_2).
	 \end{align*}
	However, the statement $\Psi_{31} \equiv \Psi_{31}(x_3)$ implies 
	\[\left|\{x_3 \in \T_3: f_{31}(x_3,\bullet) \equiv const\} \right| = |\T_3| \]
	and thus (up to symmetry) this case has already been dealt with in case 1 from above.
	The case $\Psi_{12} \equiv \Psi_{12}(x_2)$ can be handled similarly.
	Therefore we have $\Psi_{23} \equiv \Psi_{23}(x_2)$ and $\Psi_{23} \equiv \Psi_{23}(x_3)$, which implies that $\Psi_{23} \equiv const$.
\end{proof}

Finally, using the nonlinear relation satisfied by the strain components in combination with the previously derived reductions of the potentials, we obtain that the strain only depends on two out of three possible variables:

\begin{prop}
\label{prop:2D}Let $e(u): \T^3 \rightarrow \R^{3\times 3}_{sym}$ be a $\T^3$-periodic symmetrized gradient solving the differential inclusion \eqref{eq:diff_incl} and let $\Psi_{i,i+1}$ with $i\in\{1,2,3\}$ denote the functions from Lemma \ref{lem:primitive}.
	Then there exists an index $i\in \{1,2,3\}$ such that $\partial_i e \equiv 0$. Moreover, assuming without loss of generality that $i=2$, we obtain one of the decompositions
	\begin{align}\label{2D_decomposition_one}
		\begin{split}
				e_{12} & =  \phantom{\partial_2 \Psi_{23}(x_3) +{}} \partial_1 \Psi_{31}(x_3,x_1) + \dashint e_{12} \intd x,\\
				e_{23} & = \phantom{ \partial_3 \Psi_{23}(x_3) +{}} \partial_3 \Psi_{31}(x_3,x_1)   + \dashint e_{23} \intd x,\\
				e_{31} & = \partial_3 \Psi_{23}(x_3) \phantom{{}+ \partial_3 \Psi_{31}(x_3,x_1)}   + \dashint e_{31} \intd x
		\end{split}
	\end{align}
	or
	\begin{align}
	\label{2D_decomposition_two}
		\begin{split}
				e_{12} & =  \phantom{\partial_2 \Psi_{23}(x_3) +{}} \partial_1 \Psi_{31}(x_3,x_1) + \dashint e_{12} \intd x,\\
				e_{23} & = \phantom{ \partial_3 \Psi_{23}(x_3) +{}} \partial_3 \Psi_{31}(x_3,x_1)   + \dashint e_{23} \intd x,\\
				e_{31} & = \partial_1 \Psi_{12}(x_1) \phantom{{}+ \partial_3 \Psi_{31}(x_3,x_1)}   + \dashint e_{31} \intd x.
		\end{split}
	\end{align}
\end{prop}

\begin{proof}
	By symmetry and Corollary \ref{cor:primitive_constant}, we may assume $\Psi_{12} \equiv const$.
	Then the decomposition reads
	\begin{align*}
				e_{12} & =  \partial_2 \Psi_{23} + \partial_1 \Psi_{31} + \dashint e_{12} \intd x,\\
				e_{23} & = \phantom{ \partial_3 \Psi_{23} +{}} \partial_3 \Psi_{31}   + \dashint e_{23} \intd x,\\
				e_{31} & = \partial_3 \Psi_{23} \phantom{{}+ \partial_3 \Psi_{31}}   + \dashint e_{31} \intd x.
	\end{align*}
	Due to $e_{12} \in \{-1,1\}$ being a sum of two one-dimensional functions and by periodicity, for $x_3 \in \T_3$ fixed we can by Lemma \ref{lem:primitive} additionally assume
	\begin{align}
	\label{eq:vol_triv}
		\left|\{x_3 \in \T_3: \partial_2 \Psi_{23}(\bullet, x_3) \equiv 0 \} \right| >0
	\end{align}
	after exchanging the indices $1$ and $2$ if necessary.
	
	The algebraic relation $e_{31} - e_{23}e_{12} = 0$ implies
	\begin{align*}
		\partial_{3} \Psi_{23} - e_{23} \partial_2 \Psi_{23} =  e_{23} \left(\partial_1 \Psi_{31} + \dashint e_{12} \intd x \right) - \dashint e_{31} \intd x.
	\end{align*}
	Using that the right hand side of the above expression is independent of $x_2$, defining the discrete differential $\partial_2^h f(x_2):= f(x_2+h) - f(x_2)$ for $h>0$ and applying it to the above equation, we get
	\begin{align*}
		\partial_{3} \partial_2^h \Psi_{23} - e_{23} \partial_2 \partial_2^h \Psi_{23} = 0.
	\end{align*}
	Defininig
	\begin{align*}
		\partial_3 E_{23} (x_1,x_3) &  = e_{23}(x_1,x_3),\\
		E_{23} (x_1,0) &  = 0,
	\end{align*}
	we may invoke the chain rule for the Lipschitz function $\partial_2^h\Psi_{23}$ to obtain
	\begin{align*}
		\partial_{3}\left( \partial_2^h \Psi_{23}\left(x_2- E_{23}(x_1,x_3),x_3 \right) \right) = 0
	\end{align*}
	for almost all $(x_1,x_2,x_3) \in \T^3$.
		
	Recalling \eqref{eq:vol_triv}, choosing $\bar x_3 \in \T_3$ such that 
	\begin{align}\label{bar_x_3}
		\partial_2 \Psi_{23}(\bullet, \bar x_3) \equiv 0\text{ and } x_2 \mapsto \Psi_{23}(x_2,\bar x_3)\text{ is differentiable a.e.,}
	\end{align}
	and integrating in $x_3$ gives
	\begin{align*}
		 \partial_2^h \Psi_{23}\left(x_2- E_{23}(x_1,x_3),x_3 \right) = \partial_2^h \Psi_{23}\left(x_2- E_{23}(x_1,\bar x_3), \bar x_3 \right)
	\end{align*}
	almost everywhere.
	Dividing by $h \neq 0$, taking the limit $h\to 0$ and invoking the choice \eqref{bar_x_3} we see that 
	\begin{align*}
		 \partial_2 \Psi_{23}\left(x_2- E_{23}(x_1,x_3),x_3 \right)  = \partial_2 \Psi_{23}\left(x_2- E_{23}(x_1,\bar x_3), \bar x_3 \right) = 0
	\end{align*}
	for almost all $(x_1,x_2,x_3) \in \T^3$.
	
	Consequently, we have $\partial_2 \Psi_{23} \equiv 0$ and the decomposition reduces to
	\begin{align*}
				e_{12} & =  \phantom{\partial_2 \Psi_{23}(x_3) +{}} \partial_1 \Psi_{31}(x_3,x_1) + \dashint e_{12} \intd x,\\
				e_{23} & = \phantom{ \partial_3 \Psi_{23}(x_3) +{}} \partial_3 \Psi_{31}(x_3,x_1)   + \dashint e_{23} \intd x,\\
				e_{31} & = \partial_3 \Psi_{23}(x_3) \phantom{{}+ \partial_3 \Psi_{31}(x_3,x_1)}   + \dashint e_{31} \intd x,
	\end{align*}
	which is the first case \eqref{2D_decomposition_one} of the desired statement.
	
	If we had taken the other choice at the inequality \eqref{eq:vol_triv}, by symmetry, we would have deduced
	\begin{align*}
				e_{12} & = \phantom{\partial_1 \Psi_{31}(x_3) +{} } \partial_2 \Psi_{23}(x_2,x_3)  + \dashint e_{12} \intd x,\\
				e_{23} & = \partial_3 \Psi_{13}(x_3) \phantom{{}+\partial_3 \Psi_{31}(x_3,x_1)}   + \dashint e_{23} \intd x,\\
				e_{31} & =\phantom{ \partial_3 \Psi_{31}(x_3) +{} }  \partial_3 \Psi_{23}(x_3,x_2)   + \dashint e_{31} \intd x,
	\end{align*}
	so that $\partial_{1}e \equiv 0$.
	A cyclical permutation shifting the index 1 to 2 allows us to deduce the second representation \eqref{2D_decomposition_two}.

\end{proof}

\section{Proof of Theorem \ref{thm:structure}}

\label{sec:proof}

With the result of Proposition \ref{prop:2D} the argument for Theorem \ref{thm:structure} follows as in \cite[Section 4.2.1]{R16}. We repeat the argument for self-containedness.

\begin{proof}[Proof of Theorem \ref{thm:structure}]	
Since the claims of Theorem \ref{thm:structure}(i), (ii) already follow from Proposition \ref{prop:2D}, it suffices to prove the claim of Theorem \ref{thm:structure}(iii).
 Without loss of generality we may assume that as in Proposition \ref{prop:2D} we have $\p_i e = 0$ for $i=2$ as well as the decomposition \eqref{2D_decomposition_one}. Moreover, invoking Theorem \ref{thm:structure}(ii), without loss of generality, we may further assume that $e_{31}=e_{31}(x_3)$
Using the notation from \cite{R16}, we define 
\begin{align*}
v(x_1,x_3):= \Psi_{31}(x_3,x_1) +  \left(\dashint e_{12} \intd x \right)\, x_1 + \left(\dashint e_{23} \intd x \right) \, x_3.
\end{align*}
Further we set
$\Phi(s,t):= (t-E_{31}(s),s),
$
where $E_{31}'(s) = e_{31}(s)$ and $E_{31}(0)=0$. We note that the map $\Phi(s,t)$ is a bilipschitz mapping on $\R^2$. Together with the definition of $v$ we thus infer
\begin{align*}
\frac{d}{ds} v(\Phi(s,t)) = \p_3 v|_{\Phi(s,t)} - e_{31}(s) \p_1 v|_{\Phi(s,t)} = e_{23}|_{\Phi(s,t)}-e_{31}|_{\Phi(s,t)} e_{12}|_{\Phi(s,t)}= 0.
\end{align*}
As a consequence, $v(\Phi(s,t)) = \tilde{g}(t)$ for some function $\tilde{g}$ of only one variable. 
Hence,
\begin{align*}
\begin{pmatrix}
e_{12}\circ \Phi (s,t)\\
e_{23}\circ \Phi (s,t)
\end{pmatrix}
&= \begin{pmatrix}
\p_1 v \circ \Phi(s,t)\\
\p_3 v \circ \Phi(s,t)\\
\end{pmatrix}
= \begin{pmatrix} 1 &0\\ e_{31}(s) & 1
\end{pmatrix} \begin{pmatrix} \p_t (v\circ \Phi)(s,t) \\ \p_s (v\circ \Phi)(s,t) \end{pmatrix}\\
&= \begin{pmatrix} 1 & 0\\ e_{31}(s) & 1
\end{pmatrix} \begin{pmatrix} \p_t (v\circ \Phi)(s,t) \\ 0\end{pmatrix}
= \begin{pmatrix}
\tilde{g}'(t)\\
e_{31}(s) \tilde{g}'(t)
\end{pmatrix}.
\end{align*}
Defining $f_3(x_3) = e_{31}(x_3) $ and $g(t) = \tilde{g}'(t)$, then concludes the proof of the final structure result.
\end{proof}

\section*{Acknowledgements}

A.R. acknowledges funding by the SPP 2256, project ID 441068247 and support by the Heidelberg STRUCTURES Excellence Cluster which is funded by the Deutsche Forschungsgemeinschaft (DFG, German Research Foundation)
under Germany’s Excellence Strategy EXC 2181/1 - 390900948.  
T.S. acknowledges funding by the Deutsche
Forschungsgemeinschaft (DFG, German Research Foundation) under
Germany's Excellence Strategy EXC 2044 –390685587, Mathematics
Münster: Dynamics – Geometry – Structure.
Both authors would like to thank the MPI MIS where this work was initiated.

\bibliographystyle{alpha}
\bibliography{citations1}

\end{document}